\documentclass[11pt]{amsart}

\usepackage{srcltx}
\usepackage{amsmath, amssymb}
\usepackage{amscd}
\usepackage{verbatim}
\usepackage[arrow,matrix]{xy}
\usepackage{ifpdf}
\ifpdf
\usepackage{hyperref}           
\else
\usepackage[hypertex]{hyperref} 
\fi

\newtheorem{definition}{Definition}[section]
\newtheorem{theorem}[definition]{Theorem}
\newtheorem{lemma}[definition]{Lemma}
\newtheorem{corollary}[definition]{Corollary}
\newtheorem{proposition}[definition]{Proposition}
\theoremstyle{definition}
\newtheorem{remark}[definition]{Remark}

\newtheorem{example}[definition]{Example}

\newcommand\CC{\mathbb{C}}
\newcommand\NN{\mathbb{N}}
\newcommand\RR{\mathbb{R}}
\newcommand\TT{\mathbb{T}}

\newcommand\cS{\mathcal{S}}

\newcommand\spann{\mbox{\rm span}\,}

\newcommand\style{\mathcal }          

\newcommand\bh{\style B(\style H)}    

\newcommand{\B}{\style{B}}
\newcommand{\M}{M}
\renewcommand{\H}{\style{H}}
\newcommand{\K}{\style K}

\newcommand{\W}{\style{W}}

\newcommand\choquet{\partial_{\rm C}}

\newcommand\osp{{\style P}}
\newcommand\oss{{\style S}}
\newcommand\ost{{\style T}}

\newcommand\cstar{{\rm C}^*}                              
\newcommand\cstare{{\rm C}_{\rm e}^*}              

\newcommand\conv{{\rm Conv} }                          

\renewcommand\oss[1]{{\style O{}\style S}(#1)}
\newcommand\osss{{\style S}}
\renewcommand\W{\mathbb W}

\begin{document}

\title[C$^*$-Envelopes of Jordan Operator Systems]
 {C$^*$-Envelopes of Jordan Operator Systems}

\author[Mart\'\i n Argerami]{Mart\'\i n~Argerami}
\address{Department of Mathematics and Statistics\\
University of Regina\\
Regina, Saskatchewan S4S 0A2\\
Canada}
\email{argerami@math.uregina.ca}

\author[Douglas Farenick]{Douglas Farenick}
\address{%
Department of Mathematics and Statistics\\
University of Regina\\
Regina, Saskatchewan S4S 0A2\\
Canada}
\email{douglas.farenick@uregina.ca}

\thanks{This work is supported in part by the NSERC Discovery Grant program}



\begin{abstract}
We determine the boundary representations and the 
C$^*$-envelope of operator systems of the form $\mbox{\rm span}\{1,T,T^*\}$, where $T$ is a Jordan operator. 
\end{abstract}

\maketitle

\section*{Introduction}

Noncommutative Choquet theory concerns the study of boundary representations and 
C$^*$-envelopes of operator spaces, operator systems, and nonselfadjoint operator algebras 
(we refer the reader to the monographs \cite{Blecher--LeMerdy-book,Paulsen-book} and the survey
article \cite{blecher2007}).
In this paper we consider operator systems generated by a single bounded linear Hilbert space operator, which is a 
fundamental case of interest in operator theory and which continues a line of investigation that 
originates with W.~Arveson in the early 1970s \cite{arveson1972}.

The \emph{operator system generated by a bounded linear operator $T$} acting on a complex Hilbert space $\H$
is the subspace $\oss{T}\subset\B(\H)$ defined by
\[
\oss{T}\,=\,\spann\,\{1,T,T^*\}\,.
\]
Any unital representation $\rho:\cstar(T)\rightarrow\B(\H_\rho)$ of the C$^*$-algebra $\cstar(T)$ generated by the operator system $\oss{T}$ induces a unital completely positive (ucp) linear map 
$\varphi:\oss{T}\rightarrow\B(\H_\rho)$ by restricting the domain of $\rho$ to $\oss{T}$---that is, $\varphi=\rho_{\vert\oss{T}}$. Hence, $\rho$ is just one of potentially many ucp extensions of the ucp map
$\varphi:\oss{T}\rightarrow\B(\H_\rho)$ to $\cstar(T)$. 

Arveson, following the idea of the Choquet boundary in function theory, recognised the special role of those irreducible $\rho$ such that $\rho|_{\oss{T}}$ has a unique ucp extension to $\cstar(T)$.
So, a unital representation $\rho:\cstar(T)\rightarrow\B(\H_\rho)$ is a \emph{boundary representation for $\oss{T}$} if
\begin{enumerate}
\item $\rho$ is irreducible and
\item $\rho_{\vert\oss{T}}$ has a unique ucp extension to $\cstar(T)$ (namely, $\rho$ itself).
\end{enumerate}
A theorem of Arveson \cite{arveson2008} implies that
sufficiently many boundary representations exist for $\oss{T}$ in the sense that if $X=[X_{ij}]_{i,j}$ is any $p\times p$ matrix with entries $X_{ij}\in\oss{T}$, then
\[
\|X\|\,=\,\sup\left\{\| [\rho(X_{ij})]_{i,j}\|\,:\,\rho\mbox{ is a boundary representation for }\oss{T}\right\}.
\]

The ideal $\mathfrak S_{\oss{T}}$ of $\cstar(T)$ consisting of all $X\in\cstar(T)$ for which $\rho(X)=0$ for every boundary representation $\rho$ of $\oss{T}$
is called the \emph{\v Silov ideal for $\oss{T}$}. The \v Silov ideal $\mathfrak S_{\oss{T}}$
is the biggest ideal $\K$ of $\cstar(T)$ for which the canonical quotient map
$q_{\K}:\cstar(T)\rightarrow\cstar(T)/\K$ is completely isometric on $\oss{T}$. The quotient C$^*$-algebra $\cstare(\oss{T}):=\cstar(T)/\mathfrak S_{\oss{T}}$, together with
the unital completely isometric embedding of $\oss{T}$ into $\cstare(\oss{T})$ induced by the quotient homomorphism, is the C$^*$-envelope for $\oss{T}$. 
If $\oss{T}$ has a trivial \v Silov ideal, then necessarily $\cstare(\oss{T})=\cstar(\oss{T})$ and, therefore, the operator system $\oss{T}$ is said to be \emph{reduced} \cite{arveson2010}.

Although there is an extensive literature for noncommutative Choquet theory for arbitrary abstract operator systems, there remains a need for tractable interesting examples, as the 
issue of determining the boundary representations and C$^*$-envelope of a given operator system is generally quite difficult. In the case of operator systems of the form $\oss{T}$, 
there are classical function-theoretic results for normal (and subnormal) operators, and we recently considered the case of operator systems generated by
an irreducible periodic weighted unilateral shift \cite{argerami--farenick2013a}.
Motivated by the desire to develop examples in finite-dimensional Hilbert space---where the complexity of the issue is already substantial---to accompany the abstract results of \cite{arveson2010},
the purpose of this paper is to determine the boundary representations and the 
C$^*$-envelope of operator systems generated by a Jordan operator. 

The choice of Jordan operators is mostly due to the fact that there is information available about their matricial ranges. Most techniques in this paper can be extended to any operators for which their matricial ranges are known. 

A completely positive linear bijection $\varphi:\osss\rightarrow\ost$ of operator systems
is a \emph{complete order isomorphism} if $\varphi^{-1}$ is completely positive.
We
will use the following well-known elementary lemma often and without mention:
if $\varphi:\osss\to\ost$ is a completely contractive bijective linear map of operator systems such that $\varphi^{-1}$ is completely contractive, then $\varphi$ is a complete isometry. 

The dimension of $\oss{T}$ is 1,2, or 3, depending on the choice of $T$. The cases of dimensions 1 and 2 are easily determined: up to complete order isomorphism
there is exactly one operator system of dimension 1 (namely, $\CC$) and exactly one operator system of dimension 2 (namely, $\CC\oplus\CC$; see Proposition \ref{proposition: selfadjoint case}). 
However, the situation is very different for dimension 3, and in this case we are far from classifying such operator systems up to complete order isomorphism---even for operator systems acting
on finite-dimensional Hilbert spaces.

The main results on Jordan operator systems are contained in Section \S\ref{jordan op sys} and make use of the matricial range of an operator. Therefore, we begin with a preliminary section on the Choquet boundary and the
matricial range, and we point out the role of the matricial range in determining boundary representations.

\section{Choquet Boundary, Matricial Ranges, and Direct Sums of Matrix Algebras}

The set $\choquet\oss{T}$ of all boundary representations for $\oss{T}$ is called the \emph{Choquet boundary} for $\oss{T}$.  The C$^*$-envelope $\cstare(\oss{T})$ of $\oss{T}$, defined in the introduction as a quotient algebra, arises 
in a different guise, which is useful for applications:

\begin{theorem}[Arveson \cite{arveson2008}]\label{theorem: product of boundary} 
If $T\in\B(\H)$, then 
\[
\cstare(\oss{T})=\left(\prod_{\rho\in\choquet\oss{T}}\rho\right)(\cstar(\oss{T})).
\]
In particular, the map $\prod_{\rho\in\choquet\oss{T}}\rho$ is completely isometric on $\oss{T}$.
\end{theorem}

It is worth noting that although $\oss{T}$ is finite-dimensional and $\cstar(T)$ is separable, it is possible for $\oss{T}$ to
admit uncountably many boundary representations (see, for instance, Corollary 3.6 in \cite{argerami--farenick2013a}). 

\begin{definition}\label{definition:matricial range}
The \emph{matricial range} of $T\in\B(\H)$ is the set
\[
\W(T)=\bigcup_{n\in\NN}\W_n(T),
\]
where
\[
\W_n(T)=\{\varphi(T):\ \varphi:\oss{T}\to M_n(\CC)\ \mbox{ is ucp}\}
\]
\end{definition}

The matricial range is a generalisation of the classical numerical range in operator theory.
Recall that the convex set
\[
W_{\rm s}(T)\,=\,\{\langle T\xi,\xi\rangle\,:\,\xi\in\H,\;\|\xi\|=1\}
\]
is called the \emph{numerical range} of $T$. We denote the closure of $W_{\rm s}(T)$ by 
$W(T)$  and again refer to $W(T)$ as the ``numerical range'' of $T$. It is well-known that $W(T)$ is the convex compact set
\[
W(T)\,=\,\{\varphi(T)\,:\,\varphi\;\mbox{ is a state on }\oss{T}\}\,;
\]
hence, $\W_1(T)=W(T)$. Note that $W(T)=W_{\rm s}(T)$ if $\H$ has finite dimension.

\bigskip

We  denote by $\bigoplus_j M_{k_j}(\CC)$ the $\ell^\infty$-sum. The compressions 
$\pi_\ell$ of $\bigoplus_j M_{k_j}(\CC)$ to the $\ell^{\rm th}$ direct summand are irreducible representations of
$\bigoplus_j M_{k_j}(\CC)$, and if the sum is finite then these are all the irreducible representations (up to
unitary equivalence).

\begin{definition}\label{definition:irreducible family}
Let $m\in\NN\cup\{\infty\}$. The family $\{T_j:\ j\in\{1,\ldots,m\}\,\}$ of operators $T_j\in M_{k_j}(\CC)$ is an \emph{irreducible family} if $\cstar(\bigoplus_jT_j)=\bigoplus_j M_{k_j}(\CC)$.
\end{definition}

The notation $\bigoplus_{j\ne k}T_j$ will be taken to mean the operator in $\bigoplus_j M_{k_j}(\CC)$ such that the entry corresponding to $k$ in the direct sum is equal to zero. 

\begin{theorem}\label{theorem: boundary reps and matricial range}
Let $T=\bigoplus_jT_j\in\bigoplus_jM_{k_j}(\CC)$, with 
$T_j\in M_{k_j}(\CC)$ such that $\{T_j\}_j$ is an irreducible family. 
Let $\pi_\ell$ be the irreducible representation given by compression to the $\ell^{\rm th}$ block. Then the following statements are equivalent:
\begin{enumerate}
\item $\pi_\ell$ is a boundary representation for $\oss{T}$;
\item $T_\ell\not\in\W_{k_\ell}(\bigoplus_{j\ne\ell}T_j)$.
\end{enumerate}
\end{theorem}
\begin{proof}
For notational simplicity we will take $\ell=1$; this does not affect generality as we can achieve permutation of blocks by unitary conjugation, which is a complete isometry. 

Assume first that $T_1\in\W_{k_1}(\bigoplus_{j>1}T_j)$. We will show that this implies that $\pi_1$ is not boundary. By assumption there exists a ucp map $\varphi:\oss{\bigoplus_{j>1}T_j}\to M_{k_1}(\CC)$ that maps $\bigoplus_{j>1}T_j\mapsto T_1$. We can use this $\varphi$ to construct a ucp inverse to the restriction of canonical compression  $\varrho:\bigoplus_j X_j\mapsto\bigoplus_{j>1}X_j$ to $\oss{T}$. Namely, the map $\varrho':X\mapsto \varphi(X)\oplus X$ is ucp and $\varrho\varrho'(X)=X$, $\varrho'\varrho(Y)=Y$, for $X\in \oss{\bigoplus_{j>1}T_j}$, $Y\in\oss{T}$. So $\varrho$ is a complete isometry on $\oss{T}$. We can see $\varrho$ as the quotient map induced by the ideal $\K_1=M_{k_1}(\CC)\oplus0$. So we have proven that $\K_1$ is a boundary ideal for $\oss{T}$. As such, it is contained in the \v Silov ideal and thus in the kernel of any boundary representation; in particular $\pi_j(Z\oplus0)=0$ for all $j$ such that $\pi_j\in\choquet\oss{T}$ and all $Z\in M_{k_1}(\CC)$. 
As $\pi_1(Z\oplus0)=Z$ for all such $Z$, we conclude that $\pi_1$ is not a boundary representation. In other words, if $\pi_1$ is boundary then $T_1\not\in\W_{k_1}(\bigoplus_{j>1}T_j)$.

Conversely, if $\pi_1$ is not a boundary representation, then Theorem \ref{theorem: product of boundary} implies that the map induced by $T\mapsto\bigoplus_{j>1}T_j$ is completely isometric on $\oss{T}$. 
Indeed, given a boundary representation $\pi:\bigoplus_j M_{k_j}(\CC)\to\bh$, $\pi(I_{k_1}\oplus 0)$ is necessarily $I_\H$ or $0$. If it is the former, then $\dim\H=k_1$ and $\pi$ would be unitarily equivalent with $\pi_1$, a contradiction. So $\pi(Z\oplus0)=0$ for any $Z\in M_{k_1}(\CC)$ and any boundary representation $\pi$. Then Theorem \ref{theorem: product of boundary} justifies the assertion at the beginning of the paragraph. 

In conclusion, there exists a ucp inverse $\psi$ that maps $\bigoplus_{j>1}T_j\mapsto T$. Combining this with the compression to the first coordinate we get a ucp map with
\[
\bigoplus_{j>1}T_j\mapsto T\mapsto T_1,
\]
and so $T_1\in\W_{k_1}(\bigoplus_{j>1}T_j)$. We have thus shown that if $T_1\not\in\W_{k_\ell}(\bigoplus_{j>1}T_j)$, then $\pi_1$ is a boundary representation. 
\end{proof}

It is important to notice that Theorem \ref{theorem: boundary reps and matricial range} does not deal with all possible boundary representations in the case of infinite sums (see Proposition \ref{proposition: infinite sum with same eigenvalue}).

\begin{example}\label{eg1} For each $\lambda\in\CC $, let $T_\lambda\in\M_3(\CC)$ be given by
\[
T_\lambda\,=\,\left[\begin{array}{ccc} 0&1&0 \\ 0&0&0 \\ 0&0&\lambda\end{array}\right]\,.
\]
Then
\[
\cstare(\oss{T_\lambda})\,=\,  \left\{
       \begin{array}{lcl}
          \M_2(\CC)   &\;&      \mbox{if }\; |\lambda|\leq 1/2 \\
          \M_2(\CC)\oplus\CC  &\;& \mbox{if } \; |\lambda|>1/2
      \end{array}
      \right\}
      \,.
\]
\end{example}
\begin{proof}
The C$^*$-algebra generated by $\oss{T_\lambda}$ is $\M_2(\CC)\oplus\CC $. We 
have only two (classes of) irreducible representations, i.e. $\pi_1$ is compression to the upper-left $2\times2$ block, and $\pi_2$ is compression to the $(3,3)$-entry. 

One can tell right away that $\pi_1$ is a boundary representation, because the range of $\pi_2$
 is one-dimensional and thus has no room to fit the 3-dimensional operator system in. But we can also deduce the same from Theorem \ref{theorem: boundary reps and matricial range}. 
Because $\W_2(\lambda)=\{\lambda I_2\}$,
we see that
\[
\begin{bmatrix}0&1\\0&0\end{bmatrix}\not\in\W_2(\lambda),
\]
and so by Theorem \ref{theorem: boundary reps and matricial range} $\pi_1$ is a boundary representation. As \[\W_1\left(\begin{bmatrix}0&1\\0&0\end{bmatrix}\right)=B_{1/2}(0),\] the irreducible representation $\pi_2$
will be a boundary representation precisely when $\lambda\not\in B_{1/2}(0)$, i.e. when $|\lambda|>1/2$. 
\end{proof}

To conclude this section
we show below that the numerical range $W(T)$ and spectrum $\sigma(T)$
of $T$ capture information about the one-dimensional boundary representations of $\oss{T}$.
We already found in \cite{argerami--farenick2013a} that convexity plays a crucial role in understanding boundary representations. Here is more evidence of this relation:

\begin{proposition}\label{proposition: extremal spectra} Let $T\in\B(\H)$, $\lambda\in\CC$.
\begin{enumerate}
\item\label{ah1} If $\lambda=\rho(T)$ for some boundary representation $\rho$ for $\oss{T}$,
then $\lambda\in\sigma(T)\cap\partial W(T)$, and $\lambda$ is an extreme point of $W(T)$.
\item\label{ah2} Assume that $\lambda\in\sigma(T)\cap\partial W(T)$. If $\lambda$ is an extreme point of $W(T)$ and if $T$ is hyponormal,
then $\lambda=\rho(T)$ for some boundary representation $\rho$ for $\oss{T}$.
\end{enumerate}
\end{proposition}

\begin{proof} To prove \eqref{ah1}, note first that we have $\rho(T-\lambda I)=0$. As $\rho$
is multiplicative, this shows that $\lambda\in\sigma(T)$. Also, since $\rho(T)$ is scalar, we
deduce that $\rho$ is a state on $\oss{T}$, and thus $\lambda\in W(T)$. After we prove that $\lambda$
is an extreme point of $W(T)$, we will know that $\lambda\in\partial W(T)$.

Let $\varphi=\rho|_{\oss{T}}$.
Suppose that $\lambda_1,\lambda_2\in W(T)$ and that $\lambda=\frac{1}{2}\lambda_1+\frac{1}{2}\lambda_2$. As every state on
$\oss{T}$ extends to a state on $\cstar(\oss{T})$ (by the Hahn--Banach Theorem, coupled with the fact that a linear functional is positive if and only if it is unital and contractive),
there are states $\varphi_1$ and $\varphi_2$ on $\cstar(\oss{T})$ such that $\lambda_j=\varphi_j(T)$, $j=1,2$. Thus, the state
$\psi=\frac{1}{2}\varphi_1+\frac{1}{2}\varphi_2$ is an extension of $\varphi$ to $\cstar(\oss{T})$. Because $\rho$ is a
boundary representation for $\oss{T}$, $\psi=\rho$. That is,
$\rho=\frac{1}{2}\varphi_1+\frac{1}{2}\varphi_2$. But since $\rho$ is a pure state (because it is
multiplicative),
we deduce that $\varphi_1=\varphi_2=\rho$; hence, $\lambda_1=\lambda_2=\lambda$, which implies that $\lambda$ is an extreme point of $W(T)$.

For the proof of \eqref{ah2},
the hypothesis $\lambda\in\sigma(T)\cap\partial W(T)$ implies that there is a homomorphism
$\rho:\cstar(\oss{T})\rightarrow\CC $ (that is, a $1$-dimensional representation $\rho$) such that $\lambda=\rho(T)$ \cite[Theorem 3.1.2]{arveson1969}.
Assume that $\lambda$ is an extreme point of $W(T)$ and that 
$T$ is hyponormal. 
As mentioned above, the hypothesis $\lambda\in\sigma(T)\cap\partial W(T)$ implies that
there is a homomorphism $\rho:\cstar(\oss{T})\rightarrow\CC $ such that $\lambda=\rho(T)$.
Let $\varphi=\rho|_{\oss{T}}$ and define
\[
C_\varphi\,=\,\{\varphi\,:\,\varphi\;\mbox{is a state on }\cstar(\oss{T})\mbox{ such that }\varphi|_{\oss{T}}=\rho|_{\oss{T}}\}\,.
\]
The set $C_\varphi$ is evidently convex and weak*-compact. Thus, to show that  $C_\varphi$ consists of a single point it is sufficient to show that the only extreme point of $C_\varphi$ is $\rho$ itself.
To this end, select an extreme point $\varphi$ of $C_\varphi$. Because $\varphi(T)=\lambda$ is an extreme point of $W(T)$, $\varphi$ is an extremal
state on $\cstar(\oss{T})$; hence, via the GNS decomposition, there are a Hilbert space $\H_\pi$, an irreducible representation $\pi:\cstar(\oss{T})\rightarrow\B(\H_\pi)$, and
a unit vector $\xi\in\H_\pi$ such that $\varphi(A)=\langle\pi(A)\xi,\xi\rangle$ for every $A\in\cstar(\oss{T})$. In particular, $\lambda=\langle\pi(T)\xi,\xi\rangle$. Now since the
numerical range of $\pi(T)$ is a subset of the numerical range of $T$, $\lambda$ is an extreme point of $W(\pi(T))$. Moreover, as 
$T$ is hyponormal, we have that
$[\pi(T)^*,\pi(T)]=\pi\left([T^*,T]\right)$
is positive and so $W(\pi(T))$ coincides with the convex hull of the spectrum of $\pi(T)$. 
The equation $\lambda=\langle\pi(T)\xi,\xi\rangle$ and the fact that $\lambda\in \sigma\left( \pi(T) \right)\cap \partial W(\pi(T))$ imply that
$\pi(T)\xi=\lambda\xi$ and $\pi(T)^*\xi=\overline\lambda \xi$ \cite[Satz2]{hildebrandt1966}. Thus, $\varphi$ is a homomorphism and agrees with $\rho$ on the generating set $\oss{T}$;
hence, $\varphi=\rho$.
\end{proof}

It is interesting to contrast \eqref{ah1} of Proposition \ref{proposition: extremal spectra} with Theorem 3.1.2 of \cite{arveson1969}, which states that if
$\lambda\in\sigma(T)\cap\partial W(T)$, then $\lambda=\rho(T)$ for some boundary representation $\rho$ for the nonselfadjoint operator algebra
$\osp_T\subset\B(\H)$ given by the norm closure of all operators of the form $p(T)$,
for polynomials $p\in\CC\,[t]$. In this latter assertion,
there is no requirement that $\lambda$ be an extreme point of $W(T)$, and this is one way in which we see that the operator spaces $\osp_T$ and $\oss{T}$ differ fundamentally.

\section{Jordan Operator Systems}\label{jordan op sys}

We consider Jordan operators for several reasons: they are irreducible as operators in their own matrix algebras; 
they are expressed in terms of fairly simple matrices; they allow us to determine with certain ease when a family is irreducible, and we have information available about their matricial ranges.


\begin{definition} An operator $J$ on an $n$-dimensional Hilbert space $\H$ is a \emph{basic Jordan block} if there is an orthonormal
basis of $\H$ for which $J$ has matrix representation
\[
J\,=\,J_n(\lambda)\,:=\, \left[ \begin{array}{ccccc}
\lambda & 1 & 0 & \dots & 0 \\
0 & \lambda & 1 & \ddots & \vdots \\
\vdots & \ddots & \ddots & \ddots & 0 \\
\vdots & & \ddots & \ddots & 1 \\
0 & \dots & \dots& 0 & \lambda \end{array} \right]
\]
for some $\lambda\in\CC $.
\end{definition}

Note that a basic Jordan block $J=J_n(\lambda)\in\B(\H)$ is an irreducible operator. Thus, $\cstar(\oss{J})=\B(\H)$, which is a simple C$^*$-algebra.
Hence, the \v Silov boundary ideal $\mathfrak S_{\oss{J}}$ for $\oss{J}$ is trivial, which implies that $\cstare(\oss{J})=\cstar(\oss{J})=\B(\H)$. That is,
in Arveson's  terminology \cite{arveson2010}, $\oss{J}$ is a reduced operator system.

If, on the other hand, we consider the unilateral shift $S$ on $\ell^2(\NN)$, we have that $\oss{S}$ is not reduced---since $\cstare(S)=C(\TT)$, which cannot contain compact operators.

Now what is the situation if we form direct sums of basic Jordan blocks of various sizes, but with a fixed eigenvalue $\lambda$? Does the $\cstar$-envelope behave like the case of the finite or the infinite-dimensional shift?
It turns out that there is strong dichotomy, depending on how the sizes
of the blocks behave.

\begin{proposition}\label{proposition: infinite sum with same eigenvalue}
If $J=\displaystyle\bigoplus_{k=1}^\infty J_{m_k}(\lambda)\in\B\left(\ell^2(\NN )\right)$
and
$m=\sup\{m_k:\ k\in\NN\}$, then
\[
\cstare(\oss{J})=\begin{cases}C(\TT),&\mbox{  if }m=\infty\,; \\ \M_m(\CC),&\mbox{  if } m<\infty\,.\end{cases}
\]
\end{proposition}

 \begin{proof} We will assume, without loss of generality, that $\lambda=0$, because $J$ and $J-\lambda I$
 generate the same operator system.

 We consider first the case $m=\infty$.
 Recall the following positivity conditions (see, for example, \cite[Proposition 5.4]{farenick2004}):
 \begin{equation}\label{pos1}
 \alpha1_k+\beta J_k(0) +\overline\beta J_k(0)^* \,\geq\,0 \,\Longleftrightarrow\,\alpha\geq0\mbox{ and }|\beta|\leq \frac{\alpha}{2}\sec\left(\frac{\pi}{k+1}\right )
 \end{equation}
 and
 \begin{equation}\label{pos2}
 \alpha1+\beta S +\overline\beta S^* \,\geq\,0 \,\mbox{ if and only if }\,\alpha\geq0\mbox{ and } |\beta|\leq\frac{\alpha}{2}\,,
 \end{equation}
 where $S$ is the unilateral shift operator on $\ell^2(\NN )$.
 In considering $\ell^2(\NN )=\bigoplus_{k\in\NN }\ell^2(\{1,\dots,m_k\})$,
 let $P_k$ denote the projection of $\ell^2(\NN )$ onto the $k$-th direct summand
 $\CC ^{m_k}=\ell^2(\{1,\dots,m_k\})$ and define $\psi:\oss{S}\rightarrow\oss{J}$ by
 \[
 \psi(X)\,=\,\bigoplus_{k=1}^\infty P_kXP_k\,,\;X\in\oss{S}\,.
 \]
 The map $\psi$ is clearly ucp, and it sends $S$ to $J$.

 Now define $\varphi:\oss{J}\rightarrow\oss{S}$ by
 \[
 \varphi(\alpha 1+ \beta J + \gamma J^*)\,=\,\alpha 1 S + \beta S + \gamma S^*\,.
 \]
 As a linear map, we see that $\varphi^{-1}=\psi$. Hence, we need only show that $\varphi$ is completely positive.
We clearly have
 \[
 \alpha 1+ \beta J + \overline\beta J^*\,\geq\,0 \;\Longleftrightarrow\;
 \alpha1_{m_k}+\beta J_{m_k}(0) +\overline\beta J_{m_k}(0)^* \,\geq\,0 \,,\;\forall\,k\in\NN \,.
 \]
This assertion above means, by \eqref{pos1}, that
 \[
  \alpha 1+ \beta J + \overline\beta J^*\,\geq\,0 \;\Longleftrightarrow\;
  \alpha\geq0\mbox{ and }|\beta|\leq \frac{\alpha}{2}\sec\left(\frac{\pi}{m_k+1}\right ),\forall k\in\NN .
 \]
But $m=\infty$ implies that $m_k$ is arbitrarily large for some suitably chosen $k$, and so
 \[
  \alpha 1+ \beta J + \overline\beta J^*\,\geq\,0 \;\mbox{ if and only if }\;
  \alpha\geq0\mbox{ and }|\beta|\leq \frac{\alpha}{2}.
 \]
 Therefore, by \eqref{pos2},
 \[
  \alpha 1+ \beta J + \overline\beta J^*\,\geq\,0 \;\mbox{ if and only if }\;
   \alpha 1+ \beta S + \overline\beta S^*\,\geq\,0\,.
 \]
Thus, the map $\varphi:\oss{J}\rightarrow\oss{S}$ is a unital, positive map. Because $\cstare(\oss{S})=C(\TT)$, we may view,
without loss of generality, $\oss{S}$ as an
operator subsystem of $C(\TT)$. In this regard, then, the positive linear map $\varphi$ maps $\oss{T}$ into an abelian $\cstar$-algebra,
and thus $\varphi$ is automatically completely positive \cite[Theorem 3.9]{Paulsen-book}.
Hence, $\cstare(\oss{T})\,\simeq\,\cstare(\oss{S})\,=\, C(\TT)$.

Suppose now that $m<\infty$; that is, $m=\max\{m_k:\ k\in\NN\}$. We may assume without loss of generality that
$m=m_1$. Consider the quotient map
\[q:\alpha 1+\beta J+\gamma J^*\mapsto  \alpha 1_{m_1} + \beta J_{m_1}(0) + \gamma J_{m_1}(0)^*
\subset M_{m_1}(\CC).
\]
If $P_k$ is as above the compression onto the $k^{\rm th}$ block of size $m_k$, we have $P_kJ_{m_1}(0)P_k=
J_{m_k}(0)$ (this is where we use $m_k\leq m_1$). Now define
\[
\psi:\oss{J_{m_1}(0)}\rightarrow\B(\ell^2(\NN))
\]
by
\[
\psi(X)=\,\bigoplus_{k=1}^\infty P_kXP_k.
\]
The map $\psi$ is clearly ucp and, moreover, $\psi\circ q(J)=J$. This is to say that the map $q|_{\oss{J}}$ is ucp and has an ucp inverse; therefore, $q$ is
completely isometric. Thus
\[
\cstare(\oss{J})=\cstar(q(\oss{J}))=\cstar(J_{m_1}(0))=\M_{m_1}(\CC)\,,
\]
which completes the argument.
 \end{proof}

\begin{remark}
Proposition \ref{proposition: infinite sum with same eigenvalue} shows that Theorem \ref{theorem: boundary reps and matricial range} does not characterise all boundary representations. Indeed, the fact that the $\cstar$-envelope in the unbounded dimension case is $C(\TT)$ shows that there are boundary representations not coming from the $\pi_k$. 
\end{remark}

We will also need the following basic result, which is very well-known to the specialists; we have not been able, however, to find a reference in the literature. 

\begin{proposition}\label{proposition: selfadjoint case}
If $T=T^*\in\bh$, then $\cstare(\oss{T})=\CC\oplus\CC$.
\end{proposition}
\begin{proof}
Since $T$ is selfadjoint, all its irreducible representations are one-dimensional. 
Proposition \ref{proposition: extremal spectra} ensures that the only boundary representations for $\oss{T}$ are the two that send $T$ to each of the extreme points of its spectrum. By Theorem \ref{theorem: product of boundary}, we conclude that $\cstare(\oss{T})=\CC^2$.
\end{proof}

\bigskip


\begin{definition} An operator $J\in\B(\H)$ is a  \emph{Jordan operator} if $J=\displaystyle\bigoplus_{j}J_{n_j}(\lambda_j)$ for some finite or infinite sequence
of basic Jordan operators $J_{n_j}(\lambda_j)$.
\end{definition}

In the definition of Jordan operator above, we do not require the $n_j$ nor the $\lambda_j$ to be distinct. 
But we do not allow a repetition of the same pair $n_j,\lambda_j$: if we are considering a direct sum of $d$ copies of
a basic Jordan block $J_n(\lambda)$, then we denote this by $J_n(\lambda)\otimes 1_d$.

Although every operator $T$ on a finite-dimensional Hilbert space is
similar to a Jordan operator $J$, the $\cstar$-envelopes of $\oss{T}$ and $\oss{J}$ may be be quite different.
For example, the idempotent $E=\begin{bmatrix}1 & x \\ 0 & 0\end{bmatrix}$ acting on $\CC ^2$,
is similar to the orthogonal projection $P=\begin{bmatrix}1&0\\0&0\end{bmatrix}$, but
if $x\neq 0$ then $E^*E\neq EE^*$ implies that $\cstar(E)=\M_2(\CC)$, which is simple; thus,
\[
\cstare(\oss{E})\,=\,\M_2(\CC)\neq \CC \oplus\CC \,=\,\cstare(\oss{P})\,.
\]

There are a number of subtleties in attempting to determine the C$^*$-envelope of a Jordan operator $J$ on a finite-dimensional Hilbert space $\H$
in terms of the sizes of the basic Jordan blocks that combine to form $J$ and the geometry of the spectrum $J$. We have seen this already in
Example \ref{eg1}; and Proposition \ref{theorem: abelian envelope} and Remark \ref{remark: abelian envelope on finite dimensional non-commutative case}
below are further illustrations.

It is clear that when $T$ is normal, $\cstare(T)$ is abelian (being a quotient of the $\cstar$-algebra generated by $T$). The $\cstar$-envelope can also be abelian for non-normal operators: we have already mentioned that $\cstare(S)=C(\TT)$ for the unilateral shift. For finite-dimensional Jordan operators with real eigenvalues, we can characterise precisely when their $\cstar$-envelopes are abelian:

\begin{proposition}\label{theorem: abelian envelope} Assume that
$J=\displaystyle\bigoplus_{k=1}^n\left(J_{m_k}(\lambda_k)\otimes 1_{d_k}\right)$ with each pair $(m_k,\lambda_k)$ unique.
If each $\lambda_k\in\RR$, then the following statements are equivalent:
\begin{enumerate}
\item $\cstare(\oss{J})$ is abelian;
\item $m_1=\cdots=m_k=1$.
\end{enumerate}
\end{proposition}
\begin{proof}
If $m_1=\cdots=m_k=1$ then $J$ is diagonal.
It is then clear that $\cstar(\oss{J})$ is abelian and so
is any quotient of it; thus, $\cstare(\oss{J})$ is abelian (we can
reach the same conclusion by appealing to Proposition \ref{proposition: selfadjoint case},
since  in this case $J=J^*$).

Conversely, assume that $\cstare(\oss{J})$ is abelian.
Since $\cstar(\oss{J})$ is a finite direct sum
of full matrix algebras, 
its irreducible representations are precisely the compressions to the individual blocks. 
From this we conclude that every block that is not in the boundary ideal (i.e. those
corresponding to the boundary representations) has dimension 1.
Let $h_1,\ldots,h_d$ be the one-dimensional blocks that are preserved by the quotient
map $\pi$. So we have
\[
\pi(\alpha I+\beta J+\gamma J^*)=\bigoplus_{i=1}^d (\alpha+\lambda_{h_i}\beta
+{\lambda}_{h_i}\gamma)
\]
This map is completely isometric on $\cS_J$. So
we have a completely isometric inverse $\pi^{-1}$ onto $\cS_J$. Pick any $j\in\{1,\ldots,n\}$.
Let $p_j$ be the compression onto the $j^{\rm th}$ block on $\cS_J$. Then $p_j\circ\pi^{-1}$
is onto $p_j\cS_J$, and
\[
p_j\circ\pi^{-1}(\bigoplus_{i=1}^d (\alpha+\lambda_{h_i}\beta
+{\lambda}_{h_i}\gamma)
=\alpha I_j +\beta J_{m_j}(\lambda_j) + \gamma J_{m_j}(\lambda_j)^*
\]
for all $\alpha,\beta,\gamma\in\CC$. Setting $\alpha=1$, $\beta=\gamma=0$,
we have
\[
p_j\circ\pi^{-1}(\bigoplus_{i=1}^d \alpha)=\alpha I_j;
\]
now we set $\alpha=1$, $\beta=1$, and $\gamma=-1$, and we have
\[
p_j\circ\pi^{-1}(\bigoplus_{i=1}^d \alpha)=\alpha I_j+J_{m_j}(\lambda_j)-J_{m_j}(\lambda_j)^*.
\]
So it must be that $J_{m_j}(\lambda_j)=J_{m_j}(\lambda_j)^*\,$. This only happens when
$m_j=1$.
\end{proof}

\begin{remark}\label{remark: abelian envelope on finite dimensional non-commutative case}
At first sight the condition of having real eigenvalues in Proposition
\ref{theorem: abelian envelope} could be seen as a limitation of the technique
employed in the proof. This is not the case, however: consider
\[
J=\begin{bmatrix}1 \\ & \omega \\ & & \omega^2 \\ & & & 0 & 1\\ & & & 0 & 0\end{bmatrix},
\]
where $\omega=(-1-i\sqrt3)/2$. As usual, put $\oss{J}=\spann\{1,J,J^*\}$. It is easy to see
that $\cstar(\oss{J})=\CC^3\oplus \M_2(\CC)$, and so we have four Jordan blocks and four (classes of) irreducible representations.

Let
\[
U=\begin{bmatrix}1&1&1\\\omega^2&\omega&1\\ \omega&\omega^2&1\end{bmatrix}, \ \
V=\frac1{\sqrt3}\,\begin{bmatrix}1&0&0\\ 0&1&0\end{bmatrix}.
\]
Define
$\psi:\CC^3\to M_2(\CC)$ be given by
\[
\psi(\alpha,\beta,\gamma)=VU^*\,\begin{bmatrix}\alpha\\ &\beta\\ & & \gamma\end{bmatrix}\,UV^*.
\]
This map $\psi$ is linear and ucp by construction, and $\psi(1,\omega,\omega^2)=J$. By Theorem \ref{theorem: boundary reps and matricial range}, $\pi_4$ is not a boundary representation.

The other three one-dimensional
irreducible representations have to be boundary as the quotient needs to have dimension at least 3; so
 the $\cstar$-envelope of $\oss{J}$ is $\CC^3$.

\end{remark}

\begin{lemma}\label{lemma: estimate for J}
For any $\alpha,\beta,\lambda\in \CC$, $m\in\NN$, $m\geq2$,
\[
\left(|\alpha+\lambda\beta|^2+|\beta|^2\right)^{1/2}\leq\|\alpha I_m+\beta J_m(\lambda)\|\leq |\alpha+\lambda\beta|+|\beta|.
\]
\end{lemma}
\begin{proof}
Note that $\alpha 1+\beta J_m(\lambda)=(\alpha+\lambda\beta)1+\beta S_m^*$, where $S_m$ is the $m$-dimensional shift operator.
The right-hand-side estimate then follows by direct application of the triangle inequality.

For the left-hand-side inequality, take the vector $e_2\in\CC^{m}$; then
\[
((\alpha+\lambda\beta)I_m+\beta S_m^*)e_2=(\alpha+\lambda\beta)e_2+\beta e_1;
\]
the vector on the right has norm $\left(|(\alpha+\lambda\beta|^2+|\beta|^2\right)^{1/2}$, thus giving our estimate.
\end{proof}

The next proposition plays a key role in the proofs of Theorems  \ref{theorem:full classfication for Jordan} and \ref{theorem: middle blocks are non reduced when extremes are bigger}.
\begin{proposition}\label{proposition:extreme eigenvalues with bigger blocks win}
Let $J=\bigoplus_{j=1}^nJ_{m_j}(\lambda_j)$, with $\lambda_1,\ldots,\lambda_n\in\CC$ all distinct. Fix $k\in\{1,\ldots,n\}$. 
Assume that $\lambda_{s_1},\ldots,\lambda_{s_r}$ are the extreme points of $\conv\{\lambda_1,\ldots,\lambda_n\}$, and that
$\min\{m_{s_1},\ldots,m_{s_r}\}\geq m_k$.
If $\lambda_k$ is not an extreme point of $\conv\{\lambda_1,\ldots,\lambda_n\}$, then $\pi_k$ is not a boundary representation of $\oss{J}$.
\end{proposition}
\begin{proof}

Let $P_{t}:\CC^{m_{s_t}}\to\CC^{m_k}$, $t=1,\ldots,r$  be the operators defined on the canonical basis by
\[
P_{t}\,e_j=\left\{\begin{matrix}e_j& \mbox{if } j\leq m_k \\ 0 & \mbox{otherwise}\end{matrix}\right.
\]
Straightforward computations show that
\[
P_{t}\,J_{m_{s_t}}(\lambda)P_t^*=J_{m_k}(\lambda),\ \ \ \ t=1,\ldots,r\]
for any number $\lambda$ (this is where one uses the hypothesis on the sizes of the blocks). 
By hypothesis, we can find convex coefficients $a_{t}\geq0$ with $\lambda_k=\sum_{t=1}^ra_{t}\,\lambda_{s_t}$, $\sum_{t=1}^ra_{t}=1$.

Now we
define $\psi:\M_{m_{s_1}}(\CC)\oplus\cdots\oplus \M_{m_{s_r}}(\CC)\longrightarrow \M_{m_k}(\CC)$ by
\[
\psi\left(\bigoplus_{t=1}^r A_{s_t}\right)=\sum_{t=1}^ra_{t}\,P_{t}A_{s_t}P_{t}^*.
\]
It is clear that $\psi$ is ucp, since $P_{k,t}P_{k,t}^*=1_{m_k}$ and $\psi$ is made up of conjugations, convex
combinations, and direct sums.

As
\begin{align*}
\sum_{t=1}^ra_{t}\,P_{t}J_{m_{s_t}}(\lambda_{s_t})P_{t}^*
=\sum_{t=1}^ra_{t}\,J_{m_{k}}(\lambda_{s_t}) 
= J_{m_k}(\sum_{t=1}^ra_{t}\,\lambda_{s_t}) 
&= J_{m_k}(\lambda_k),
\end{align*}
we have $\psi\left(\bigoplus_{t=1}^r J_{m_{s_t}}(\lambda_{s_t})\right)=J_{m_k}(\lambda_k)$. Thus \[
J_{m_k}(\lambda_k)\in\W_{m_k}\left(\bigoplus_{t=1}^r J_{m_{s_t}}(\lambda_{s_t})\right).
\] By Theorem \ref{theorem: boundary reps and matricial range}, $\pi_k$ is not a boundary representation. 
\end{proof}

We are now in position to determine the boundary representations of the operator system generated by a finite-dimensional Jordan operator with real eigenvalues. 

\begin{theorem}\label{theorem:full classfication for Jordan}
Let $J=\bigoplus_{j=1}^nJ_{m_j}(\lambda_j)$, with $\lambda_1>\cdots>\lambda_n\in\RR$. If $k\in\{2,\ldots,n-1\}$, then the following statements are equivalent:
\begin{enumerate}
\item\label{theorem:full classfication for Jordan:1} $\pi_k$ is a boundary representation of $\oss{J}$;
\item\label{theorem:full classfication for Jordan:2} At least one of the following holds:
\begin{enumerate}
\item $\max\{m_1,\ldots,m_{k-1}\}<m_k$;
\item $\max\{m_{k+1},\ldots,m_n\}<m_k$.
\end{enumerate}
\end{enumerate}
If $k\in\{1,n\}$, then the following statements are equivalent:
\begin{enumerate}\setcounter{enumi}{2}
\item\label{theorem:full classfication for Jordan:3} $\pi_k$ is a boundary representation of $\oss{J}$;
\item\label{theorem:full classfication for Jordan:4} One of the following holds:
\begin{enumerate}
\item $m_k>1$;
\item $m_k=1$ and $\lambda_k\not\in W(J_j(\lambda_j))$, for all $j\in\{1,\ldots,n\}\setminus\{k\}$.
\end{enumerate}
\end{enumerate}
\end{theorem}
\begin{proof}
\eqref{theorem:full classfication for Jordan:1}$\implies$\eqref{theorem:full classfication for Jordan:2} If \eqref{theorem:full classfication for Jordan:2} fails, then we are in the conditions of Proposition \ref{proposition:extreme eigenvalues with bigger blocks win} (if necessary, by considering an appropriate subsystem of $\oss{J}$) and so $\pi_k$ is not boundary. 

\eqref{theorem:full classfication for Jordan:2}$\implies$\eqref{theorem:full classfication for Jordan:1} Theorem 1 in \cite{haagerup-delaharpe1992} implies that 
\[
\lambda_j + i\,\cos\frac\pi{m_j+1}\in\W_1(J_{m_j}(\lambda_j))\subset B_{\cos\frac\pi{m_j+1}}(\lambda_j)
\]
(Haagerup and de la Harpe prove that $\cos\pi/(n+1)\in\W_1(J_n(0))$; as $i\,J_n(0)$ is unitarily equivalent to $J_n(0)$, one can construct a state $\varphi$ with $\varphi(J_n(\lambda))=\lambda+i\,\cos\pi/(n+1)$). 

Note that $\lambda_{k}+i\,\cos\frac\pi{m_k+1}\not\in\conv\bigcup_{j\ne k}B_{\cos\frac\pi{m_j+1}}(\lambda_j)$. Indeed, by hypothesis all points in the convex combination will have imaginary part less than $\max\{\cos\frac\pi{m_j+1}:\ j\ne k\}<\cos\frac\pi{m_k+1}$.

This implies that $J_{m_k}(\lambda_k)\not\in\W_{m_k}(\bigoplus_{j\ne k}J_{m_j}(\lambda_j))$ (otherwise, evaluating on states would contradict the previous paragraph). By Theorem \ref{theorem: boundary reps and matricial range}, $\pi_k$ is a boundary representation. 

\eqref{theorem:full classfication for Jordan:3}$\implies$\eqref{theorem:full classfication for Jordan:4} We will assume that $\pi_1$ is a boundary representation (the argument for $\pi_n$ is entirely similar). If $m_1=1$, then by Theorem \ref{theorem: boundary reps and matricial range} we have 
\[
\lambda_1\not\in W(\bigoplus_{j\ne1}J_{m_j}(\lambda_j))=\conv\bigcup_{j\ne1}W(J_{m_j}(\lambda_j)).
\]
In particular, $\lambda_1$ fails to be in each of the individual numerical ranges.

\eqref{theorem:full classfication for Jordan:4}$\implies$\eqref{theorem:full classfication for Jordan:3} If $m_k=1$ and  $\lambda_k\not\in W(J_{m_j}(\lambda_j))$, for all $j\in\{1,\ldots,n\}\setminus\{k\}$, then $\lambda_k\not\in\conv\bigcup_{j\ne k}W(J_{m_j}(\lambda_j))$; indeed, since $\lambda_k$ is at the extreme of the list $\lambda_1,\ldots,\lambda_n$, if $\lambda_k$ were in the convex hull then there would exist a fixed $j\ne k$ with $\lambda_1\in W(\lambda_{m_j}(\lambda_j))$ (these are all discs with centre on the real line); as it is not, it fails to be in the convex hull. Then Theorem \ref{theorem: boundary reps and matricial range} guarantees that $\pi_k$ is a boundary representation. 

Assume now that $m_k>1$. If $m_j=1$ for all $j\ne k$, then $W(J_{m_k}(\lambda_k))\not\subset W(\bigoplus_{j\ne k}J_{m_j}(\lambda_j))$ as the former is a ball and the latter a line. So in this case Theorem \ref{theorem: boundary reps and matricial range} implies that $\pi_k$ is a boundary representation. 
Otherwise, consider first the case $k=1$, i.e. $m_1>1$. 
Note that we can assume without loss of generality that $\lambda_n>0$ (since translating $J$ by a scalar multiple of the identity will still generate the same operator system, and eigenvalues and matricial ranges respect affine transformations). We will also assume that $\lambda_2^2+2\lambda_2<\lambda_1^2$; this can be achieved by multiplying $J$ by $4\lambda_2/(\lambda_1^2-\lambda_2^2)$, again without changing $\oss{J}$. These changes can alternatively be done by, instead of translating and scaling $J$, considering below---instead of the norm of $J$---the norm of $\frac{2(\lambda_2+c)}{(\lambda_1+c)^2+(\lambda_2+c)^2}\,(J+c\,I)$ for some $c>|\lambda_n|$.

Let $\pi$ be the representation $\pi:\bigoplus_{j}J_{m_j}(\lambda_j)\mapsto\bigoplus_{j\ne1}J_{m_j}(\lambda_j)$. If $\pi_1$ were not boundary, then this map would be completely isometric. We have, by Lemma \ref{lemma: estimate for J},
\[
\|J\|=\max_j\|J_{m_j}(\lambda_j)\|\geq\max_j(\lambda_j^2+1)^{1/2}=(\lambda_1^2+1)^{1/2}.
\]
On the other hand, again by Lemma \ref{lemma: estimate for J},
\[
\|\pi(J)\|=\max_{j\ne1}\|J_{m_j}(\lambda_j)\|\leq\max_{j\ne1}|\lambda_j|+1=\lambda_2+1
\]
So we have
\[
\|\pi(J)\|\leq\lambda_2+1=(\lambda_2^2+2\lambda_2+1)^{1/2}<(\lambda_1^2+1)^{1/2}\leq\|J\|.
\]
So $\pi$ is not completely isometric, and thus $\pi_1$ is a boundary representation. 

The case $k=n$ can be proven with the same method, by switching first from $J$ to $-J$. 
\end{proof}

\begin{remark}
The ideas in Theorem \ref{theorem:full classfication for Jordan} can certainly be applied to cases where the $\lambda_j$ are allowed to be complex. But the possibilities seem 
much harder to consider---as the example in Remark \ref{remark: abelian envelope on finite dimensional non-commutative case} already illustrates. Note also that for complex eigenvalues one has no control over the positions of the balls considered in the proof of \eqref{theorem:full classfication for Jordan:2}$\implies$\eqref{theorem:full classfication for Jordan:1} in Theorem \ref{theorem:full classfication for Jordan}.
\end{remark}

In the concrete case where eigenvalues are real and the blocks corresponding to the extreme eigenvalues are the biggest, we can calculate the $\cstar$-envelope very explicitly. 

\begin{corollary}\label{theorem: middle blocks are non reduced when extremes are bigger}
If
$
J=\displaystyle\bigoplus_{k=1}^n\left(J_{m_k}(\lambda_k)\otimes 1_{d_k}\right)$,
where $\lambda_1>\lambda_2>\cdots>\lambda_n$ are real, and if
$\max\{m_2,\ldots,m_{n-1}\}\leq\min\{m_1,m_n\}$, then
\[
\cstare(\oss{J})=\begin{cases} \M_{m_1}(\CC)\oplus \M_{m_n}(\CC)&\mbox{ if }\min\{m_1,m_n\}\geq2 \\
\CC\oplus\CC&\mbox{ if }m_1=m_n=1\\
\M_{m_n}(\CC)&\mbox{ if }m_1=1,\ m_n\geq2,\ |\lambda_1-\lambda_n|\leq\cos\frac\pi{(m_n+1)} \\
\CC\oplus \M_{m_n}(\CC)&\mbox{ if }m_1=1,\ m_n\geq2,\ |\lambda_1-\lambda_n|>\cos\frac\pi{(m_n+1)} \\
\M_{m_1}(\CC)&\mbox{ if }m_1\geq2,\ m_n=1,\ |\lambda_1-\lambda_n|\leq\cos\frac\pi{(m_1+1)} \\
\M_{m_1}(\CC)\oplus\CC&\mbox{ if }m_1\geq2,\ m_n=1,\ |\lambda_1-\lambda_n|>\cos\frac\pi{(m_1+1)} \\
\end{cases}
\]
\end{corollary}
\begin{proof}
 It is easy to see that
\[
\cstar(\oss{J})=\bigoplus_{k=1}^n\left(M_{m_k}(\CC) \otimes 1_{d_k}\right)
\]
It is also clear that $\oss{J}$ is completely order isomorphic to the operator system generated by
\begin{equation}\label{eq:J without multiplicities}
J'=\bigoplus_{k=1}^nJ_{m_k}(\lambda_k),
\end{equation}
and so we can eliminate the multiplicities from our computations. Thus, without loss of generality, we assume that
$J$ is of the form \eqref{eq:J without multiplicities}, and $\cstar(\oss{J})=\bigoplus_{k=1}^n\M_{m_k}(\CC)$; this, because the condition on the eigenvalues guarantees that the family is irreducible.

By Theorem \ref{theorem:full classfication for Jordan}, the only possible boundary representations are $\pi_1$ and  $\pi_n$. 

\noindent\textbf{Case 1: }$m_1>1$, $m_n>1$. Here Theorem \ref{theorem:full classfication for Jordan} guarantees that both $\pi_1$ and $\pi_n$ are boundary representations.

\noindent\textbf{Case 2: }$m_1=m_n=1$. We are in the situation of Proposition \ref{proposition: selfadjoint case}, so
$\cstare(\oss{J})=\CC\oplus\CC$ (i.e. both $\pi_1$ and $\pi_n$ are boundary representations).

\noindent\textbf{Case 3: }$m_1=1$, $m_n\geq2$, $|\lambda_1-\lambda_n|\leq\cos\pi/(m_n+1)$. Note that this
last condition is the same as $\lambda_1\in W(J_{m_n}(\lambda_n))$ (see \cite{haagerup-delaharpe1992}). 
Then Theorem \ref{theorem:full classfication for Jordan} implies that $\pi_1$ is not boundary.

\noindent\textbf{Case 4: }$m_1=1$, $m_n\geq2$, $|\lambda_1-\lambda_n|>\cos\pi/(m_n+1)$. So $\lambda_1\not\in W(J_{m_n}(\lambda_n))$.
The condition $m_1=1$ implies that $m_j=1$ for all $j\ne n$, and in particular $W(J_{m_j}(\lambda_j))=\{\lambda_j\}$ for all $j\ne n$. 
 We are in a
situation similar to Case 3, but in this case $\lambda_1\not\in W(\bigoplus_{j\ne1}J_{m_j}(\lambda_j))$. So Theorem \ref{theorem:full classfication for Jordan} implies that $\pi_1$ is boundary.

\noindent\textbf{Cases 5 and 6: } $m_1\geq2$, $m_n=1$. We did not  use that $\lambda_1>\lambda_n$
in Cases 3 and 4 (only that it was at the extreme of the list), so the same proofs apply with the roles of $1$ and $n$ reversed.
\end{proof}

In trying to classify the irreducible representations of a singly generated operator system of the form $\oss{T}$ with $T=\bigoplus_jT_j$ for an irreducible family, recall that  Theorem \ref{theorem: boundary reps and matricial range} gives us a characterisation of the boundary representations, namely $\pi_\ell$ is a boundary representation  if and only $T_\ell\not\in\W_{m_\ell}(\bigoplus_{j\ne\ell}T_j)$. So in principle one could go testing this condition starting with $T_1$, then $T_2$, etc., and determining which blocks do not correspond to boundary representations. After ``erasing'' those blocks we end up with a reduced operator system. But how can we be sure that if we perform this procedure in any order we will obtain the same result? After all, one could imagine that $T_1\in\W(\bigoplus_{j\geq2}T_j)$ in a way that the ucp map that realises $T_1$ depends essentially on $T_2$; and that $T_2\in\W(\bigoplus_{j\ne2}T_j)$ in a way that the ucp map that realises $T_2$ depends essentially on $T_1$. Is there a contradiction? We show below that no contradiction arises:

\begin{proposition}\label{proposition: coherence}
Let $T_1,\ldots,T_n\in\bh$ such that $T_1\in\W(\bigoplus_2^n T_j)$, $T_2\in\W(\bigoplus_{j\ne2}T_j)$. Then
\[
\oss{\bigoplus_1^nT_j}\simeq\oss{\bigoplus_2^nT_j}\simeq\oss{\bigoplus_{j\ne2}T_j}.
\]
\end{proposition}
\begin{proof}
By hypothesis there exists a ucp map \[\varphi:\oss{\bigoplus_2^n T_j}\to\oss{T_1}
\] with $\varphi(\bigoplus_2^nT_j)=T_1$.
Let 
\[
P:\oss{\bigoplus_1^n T_j}\to\oss{\bigoplus_2^n T_j}
\]
be the compression map, i.e. $P(\bigoplus_1^nX_j)=\bigoplus_2^nX_j$, and let 
\[
Q:\oss{\bigoplus_2^n T_j}\to\oss{\bigoplus_1^n T_j}
\]
be the map $X\mapsto \varphi(X)\oplus X$. As both $P$ and $Q$ are ucp and they map $\bigoplus_1^nT_j$ to itself, we have that $Q\circ P$ is the identity on $\oss{\bigoplus_1^n T_j}$. So both $P$, $Q$ are completely isometric, and we get the isomorphism 
$\oss{\bigoplus_1^nT_j}\simeq\oss{\bigoplus_2^nT_j}$. The other isomorphism is obtained in the same way. 
\end{proof}

\begin{remark}\label{remark: the impossible isomorphisms}
Note that one need not have the isomorphism with $\oss{\bigoplus_3^nT_j}$ in Proposition \ref{proposition: coherence}. For instance, let
\[
T_1=\begin{bmatrix}0&1\\0&0\end{bmatrix},\ T_2=\frac1{2}\,\begin{bmatrix}1&-1\\1&-1\end{bmatrix},\ \ T_3=1
\]
(note that $T_1,T_2$ are unitary conjugates of each other). 
Then $\oss{T_1\oplus T_2\oplus T_3}\simeq\oss{T_1}\simeq\oss{T_2}\not\simeq\oss{T_3}$.
\end{remark}
The last isomorphism can occur in adequate examples, as shown below. We will also address the issue that in the conditions of Proposition \ref{proposition: coherence}, there is no reason to expect that $\oss{T_1}\simeq\oss{T_2}$. Indeed, let
\[
T=1\oplus
\begin{bmatrix}1/2&1\\0&1/2\end{bmatrix}\oplus\begin{bmatrix}0&1\\0&0\end{bmatrix}\oplus
\begin{bmatrix}2&1\\0&2\end{bmatrix}.
\] 
Then Proposition \ref{proposition:extreme eigenvalues with bigger blocks win} guarantees that the first two blocks are in the matricial range of the last two, so by Theorem \ref{theorem:full classfication for Jordan}
\[
\oss{T}\simeq\oss{\begin{bmatrix}0&1\\0&0\end{bmatrix}\oplus
\begin{bmatrix}2&1\\0&2\end{bmatrix}}.
\] The first two blocks clearly generate non-isomorphic operator systems, as the first one will have dimension 1, and the second dimension 3. 

\section{Some Examples}

We show below some examples where one uses the results above to decide whether a given operator system generated by a Jordan operator is reduced.

\begin{example}\label{remark: sizes do matter} If
\[
J=\begin{bmatrix}0 \\ & 1&1\\ &0&1\\ & & & 2\end{bmatrix}\,,
\]
then $\oss{J}$
is reduced. 
To verify that $\oss{J}$ is indeed reduced, it is enough to look at the combined matricial ranges. Note first that $\pi_2$ is certainly a boundary representation, because if it were not Theorem \ref{theorem: product of boundary} would make the $\cstar$-envelope either $\CC$ or $\CC\oplus\CC$, which cannot contain a 3-dimensional subspace (or we can use Theorem \ref{theorem:full classfication for Jordan} and notice that $B_{1/2}(1)$ contains neither $0$ or $2$; or Theorem \ref{theorem: boundary reps and matricial range} and notice that the numerical range of the direct sum $0\oplus2$ is the segment $[0,2]$ that contains no ball). We have
\[
2\not\in W\left(0\oplus\begin{bmatrix}1&1\\ 0&1\end{bmatrix}\right)
=\conv\left\{W(0)\cup W\left(\begin{bmatrix}1&1\\ 0&1\end{bmatrix}\right)\right\}
=\conv\{0\cup B_{1/2}(1)\},
\]
so $\pi_3$ is a boundary representation. And
\begin{align*}
0&\not\in W\left(\begin{bmatrix}1&1\\ 0&1\end{bmatrix}\oplus2\right)
=\conv\left\{W\left(\begin{bmatrix}1&1\\ 0&1\end{bmatrix}\right)\cup W(2)\right\}\\
&=\conv\{B_{1/2}(1),2\},
\end{align*}
so $\pi_1$ is boundary.

\end{example}

\begin{example}\label{example: reduced with bigger blocks in the middle} For the Jordan operator
$J=J_1(3)\oplus J_2(2)\oplus J_2(1)\oplus J_1(0)$,
the operator system $\oss{J}$ is reduced.

Again we look at the numerical ranges. We have 
\begin{align*}
W(J_1(3))=\{3\},\ \ W(J_2(2))=B_{1/2}(2), \\ W(J_2(1))=B_{1/2}(1),\ \ W(J_1(0))=\{0\}.
\end{align*}
It is easy to check that none of the four sets is in the convex hull of the other three. So none of the four components of $J$ is in the matricial range of the other three; by Theorem \ref{theorem: boundary reps and matricial range} every irreducible representation is boundary, i.e. $\oss{J}$ is reduced. 
\end{example}

\begin{example}
(Compare with Example \ref{example: reduced with bigger blocks in the middle}) With the Jordan
operator $J=J_1(3)\oplus J_2(2)\oplus J_2(1/2)\oplus J_1(0)$, the operator system $\oss{J}$ is
\textbf{not} reduced. Indeed, $W(J_2(1/2)$ is the disk of radius $1/2$ centered
at $1/2$, and so $0\in W(J_2(1/2))$. By Theorem \ref{theorem:full classfication for Jordan}, $\pi_4$ is not boundary.
\end{example}



\begin{thebibliography}{10}

\bibitem{argerami--farenick2013a}
M.~Argerami and D.~Farenick.
\newblock The {$C^*$}-envelope of an irreducible periodic weighted unilateral
  shift.
\newblock {\em Int. Eq. and Op. Th.}, in press, 2013.

\bibitem{arveson1969}
W.~Arveson.
\newblock Subalgebras of {$C\sp{\ast} $}-algebras.
\newblock {\em Acta Math.}, 123:141--224, 1969.

\bibitem{arveson1972}
W.~Arveson.
\newblock Subalgebras of {$C\sp{\ast} $}-algebras. {II}.
\newblock {\em Acta Math.}, 128(3-4):271--308, 1972.

\bibitem{arveson2008}
W.~Arveson.
\newblock The noncommutative {C}hoquet boundary.
\newblock {\em J. Amer. Math. Soc.}, 21(4):1065--1084, 2008.

\bibitem{arveson2010}
W.~Arveson.
\newblock The noncommutative {C}hoquet boundary {III}: operator systems in
  matrix algebras.
\newblock {\em Math. Scand.}, 106(2):196--210, 2010.

\bibitem{blecher2007}
D.~P. Blecher.
\newblock Positivity in operator algebras and operator spaces.
\newblock In {\em Positivity}, Trends Math., pages 27--71. Birkh\"auser, Basel,
  2007.

\bibitem{Blecher--LeMerdy-book}
D.~P. Blecher and C.~Le~Merdy.
\newblock {\em Operator algebras and their modules---an operator space
  approach}, volume~30 of {\em London Mathematical Society Monographs. New
  Series}.
\newblock The Clarendon Press Oxford University Press, Oxford, 2004.
\newblock Oxford Science Publications.

\bibitem{farenick2004}
D.~R. Farenick.
\newblock Pure matrix states on operator systems.
\newblock {\em Linear Algebra Appl.}, 393:149--173, 2004.

\bibitem{haagerup-delaharpe1992}
U.~Haagerup and P.~de~la Harpe.
\newblock The numerical radius of a nilpotent operator on a {H}ilbert space.
\newblock {\em Proc. Amer. Math. Soc.}, 115(2):371--379, 1992.

\bibitem{hildebrandt1966}
S.~Hildebrandt.
\newblock \"{U}ber den numerischen {W}ertebereich eines {O}perators.
\newblock {\em Math. Ann.}, 163:230--247, 1966.

\bibitem{Paulsen-book}
V.~Paulsen.
\newblock {\em Completely bounded maps and operator algebras}, volume~78 of
  {\em Cambridge Studies in Advanced Mathematics}.
\newblock Cambridge University Press, Cambridge, 2002.

\end{thebibliography}
\end{document}